\documentclass[a4paper]{IEEEtran}

\usepackage{amssymb}
\usepackage{amsmath}	
\usepackage{mathtools}
\usepackage[utf8]{inputenc}
\usepackage[english]{babel}
\usepackage{amsthm}
\usepackage{tikz,pgf,pgfplots}
\usepackage[utf8]{inputenc}

\theoremstyle{plain}
\newtheorem{lemma}{Lemma}
\newtheorem{proposition}{Proposition}
\newtheorem{theorem}{Theorem}
\newtheorem{corollary}{Corollary}

\theoremstyle{definition}
\newtheorem{definition}{Definition}

\theoremstyle{remark}
\newtheorem{remark}{Remark}

\title {Analysis of Convergence for the Newton Method in DC Microgrids}

\author{\IEEEauthorblockN{Alejandro Garc\'es,} \textit{Senior Member IEEE}

\thanks{Preprint submitted to IEEE Transactions on Power Systems}}

\begin{document}
\maketitle
\begin{abstract}
The power flow is a non-linear problem that requires a Newton's method to be solved in dc microgrids with constant power terminals.  This paper presents sufficient conditions for the quadratic convergence of the Newton's method in this type of grids. The classic Newton method as well as an approximated Newton Method are analyzed in both master-slave and island operation with droop controls. Requirements for the convergence as well as for the existence  and uniqueness of the solution starting from voltages close to 1pu are presented.  Computational results complement this theoretical analysis.  
\end{abstract}


\section{Introduction}

\IEEEPARstart{M}{icrogrids} promise to play a fundamental role in the future of the smart grid concept \cite{WileyReviewMicroGrids,josef}.  In particular, dc microgrids are gaining increasing interest due to their advantages in terms of efficiency, reliability and controlability. A dc microgrid allows high efficiency and simplified control due to absence of reactive power or frequency controls; it allows high reliability due to its capability of island operation; it permits simple integration since many generation and storage technologies are already dc (i.g solar photovoltaic, batteries)\cite{Estandarization,Qualitative_topology_effects}.  In addition, most of the home appliances could be adapted to operate in dc \cite{dc_come_home}.

In a typical dc microgid, power electronics converters can be operated as constant current or as constant power. In the later case, the model of the grid becomes non-linear and requires a power flow algorithm for stationary state analysis \cite{adhoc}.  The problem is non-linear/non-convex and requires to be solved by using numerical algorithms \cite{review_dcmicrogirds}.  Of course, convergence is not always guaranteed in these type of algorithms due to the non-linear nature of the problem.  An algorithm could even diverge or converge to a non-realistic solution. Consequently, it is necessary to establish the exact conditions in which a power flow algorithm converges to a unique and realistic solution. 

On the other hand, there are two main type of controls in dc microgrids namely, master-slave and droop control.  In master-slave control, a converter fixes the voltage of the entire grid; this is the most usual operation for grid connected microgrids.  For island-mode operation, the power in the grid is modified by a droop control in order to achieve an stable equilibrium point.  Both operation modes require a load flow algorithm \cite{dc_powerflow_energyconversion}.

In this paper we analyze the convergence of the Newton's method as well as an approximated Newton's method for the power flow in dc microgrids.  This analysis is important for two main reasons: first, the power flow algorithm requires to be executed many times in both, operation and planning of microgrids.  In operation, a guaranteed convergence is a desired feature in the context of the smart-grids where human supervision is less usual.  In planning, the power flow could be part of other algorithms, specially in heuristic optimization problems \cite{genetico}. Hence, quadratic convergence and uniqueness of the solution are key conditions.  Second, the power flow gives the equilibrium point of the dynamical model of the grid.  Finding the equilibrium is the fist step in most of the studies related to dynamics and stability of microgrids \cite{koguiman}.  

Convergence of the Gauss-Seidel  method was recently analyzed by the author  for master-slave operation \cite{GARCES2017149}. Here, we extend this result by defining exact conditions for the convergence of the Newton's method in both master-slave and island operation. Moreover, the convergence of an approximated Newton's method is analyzed. This method is similar to the fast decoupled load flow for ac grids. We use the Kantorovitch's theory for the former and the contraction mapping theory for the later.  Despite being a classic results in real analysis, these theories have not been used before to analyze these problems.  As expected, the Newton method has quadratic convergence although in a small basin of attraction whereas the approximated Newton's method has a guaranteed linear convergence.

Notice there is a linearization in ac grids which is also named dc-power flow.  That name comes from an analogy between angles in ac grids and voltages in linear dc grids \cite{not_related1}. This paper is not related to these type of analogies or linearizations.  We are interested in grids that are actually dc and non-linear due to the presence of power converters.

The rest of the paper is organized as follows: Section II describes the model of the grid in both, master-slave and island operation. The Newton's method is analized in Section II, followed by the analysis of the approximated Newton's method in Section III.  After that, numerical simulations are performed followed by conclusions and references.  

\section{Grid Model}

\subsection{Master-slave operation}

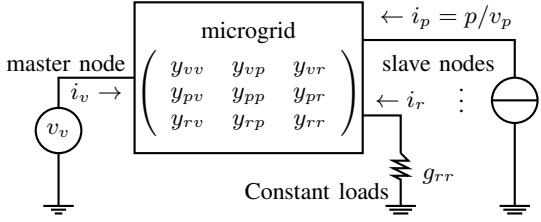
\begin{figure}[tb]
\centering
\small
\begin{tikzpicture}[x=1mm,y=1mm]

\draw[thick] (0,0) rectangle +(30,20);
\node[] at (15,16) {microgrid};
\node[] at (15,8) {$\left(\begin{array}{ccc}
y_{vv} & y_{vp} & y_{vr} \\
y_{pv} & y_{pp} & y_{pr} \\
y_{rv} & y_{rp} & y_{rr} \\
\end{array}\right)
$};

\draw[thick] (-10,-7) -- (-10,0);
\draw[thick] (-12,-7) -- (-8,-7);
\draw[thick] (-11,-7.5) -- (-9,-7.5);
\draw[thick] (-10.5,-8) -- (-9.5,-8);
\draw[thick] (-10,3) circle (3);
\draw[thick] (-10,6) |- (0,10);
\node at (-10,3) {$v_{v}$};
\node at (-5,8) {$i_{v}\rightarrow$};
\node at (-9,12) {master node};

\draw[thick] (30,15) -| (50,10);
\draw[thick] (50,7) circle (3);
\draw[thick] (47,7) -- +(6,0);
\draw[thick] (50,4) -- (50,-7);
\draw[thick] (48,-7) -- (52,-7);
\draw[thick] (48.5,-7.5) -- (51.5,-7.5);
\draw[thick] (49,-8) -- (51,-8);
\node at (41,18) {$\leftarrow i_{p}=p/v_{p}$};
\node at (43,8){$\vdots$};
\node at (40,12) {slave nodes};

\draw[thick] (30,5) -| (35,0);
\draw[thick] (35,0) -- +(1,-1) -- +(-1,-1) -- +(1,-2) -- +(-1,-2) -- +(1,-3) -- +(-1,-3) -- +(0,-3.5) -- +(0,-7);
\draw[thick] (33,-7) -- (37,-7);
\draw[thick] (33.5,-7.5) -- (36.5,-7.5);
\draw[thick] (34,-8) -- (36,-8);
\node at (24,-5) {Constant loads};
\node at (40,-3) {$g_{rr}$};
\node at (35,7) {$\leftarrow i_{r}$};
\end{tikzpicture}
\caption{Schematic representation of a DC microgrid with three types of terminals: constant voltage $v_v$, constant power $i_p$ and constant resistance $i_r$}
\label{fig:esquema_microgrid}
\end{figure}

Let us consider a dc microgrid with master-slave operation represented by its nodal admittance matrix $y$ and three types of terminals namely: constant voltage (master node), constant admittance and constant power as depicted in Fig \ref{fig:esquema_microgrid}; each of these terminals are represented by sub-indices $v,r$ and $p$ respectively.   The model of the grid is given by 

\begin{equation*}
\left(\begin{array}{c}
i_v \\
i_p \\
i_r \\
\end{array}\right)
 = 
\left(\begin{array}{ccc}
y_{vv} & y_{vp} & y_{vr} \\
y_{pv} & y_{pp} & y_{pr} \\
y_{rv} & y_{rp} & y_{rr} \\
\end{array}\right)
\left(\begin{array}{c}
v_v \\
v_p \\
v_r
\end{array}\right)
\end{equation*}

Constant admittance terminals can be represented by a diagonal matrix $g_{rr}$ such that $i_r=-g_{rr}v_r$ (the sign comes from the direction of the current).  Therefore, we can collect the terms of the admitance matrix for a Kron's reduction as follows

\begin{equation*}
\left(\begin{array}{c}
i_{v} \\
i_{p} \\\hline
0 \\
\end{array}\right)
 = 
\left(\begin{array}{cc|c}
y_{vv} & y_{vp} & y_{vr} \\
y_{pv} & y_{pp} & y_{pr} \\\hline
y_{rv} & y_{rp} & y_{rr}+g_{rr} \\
\end{array}\right)
\left(\begin{array}{c}
v_v \\
v_p \\\hline
v_r
\end{array}\right)
\end{equation*}

Let us define new matrices $Y_{pv}$ and $Y_{pp}$ as follows

\begin{eqnarray*}
Y_{pv} = y_{pv} - y_{pr}(y_{rr}+g_{rr})^{-1}y_{rv} \\
Y_{pp} = y_{pp} - y_{pr}(y_{rr}+g_{rr})^{-1}y_{rp     } 
\end{eqnarray*}

\noindent then

\begin{equation*}
i_p = Y_{pv} v_v + Y_{pp}v_p
\end{equation*}

The power flow consists in finding the state variables (or independent variables) which in this case is the vector $v_p$, since $v_v$ is already known and the other quantities, such as the power flows and power losses, can be easily computed from these voltages. In the following, we will drop the sub index of $v_p$ in order to simplify the notation.   

On the other hand, constant power terminals introduce a non-linear vector function $G:\mathbb{R}^n\times\mathbb{R}^n\rightarrow\mathbb{R}^n$ that represents the currents $i_p$ as function of their voltages  $v_p$ and the controlled power $P$ as follows:

\begin{equation*}
	i_{p}=G(v) = diag(v)^{-1} P
\end{equation*}

in other words, each current is given by $i = p/v$.  Since $P$ and $v_{v}$ are known, the resulting non-linear system is the following

\begin{equation}
	F(v) = G(v) - Y_{pv} v_v - Y_{pp} \cdot v = 0
	\label{eq:loadflow}
\end{equation}

This is a system of non-linear equations that requires to be solved by a Newton's method.

\subsection{Island operation}

In island operation the master terminal is disconnected meaning that $i_v=0$; hence, the node $v$ can be eliminated by a new Kron's reduction \cite{kron} as follows

\begin{equation*}
	Y_{s} = Y_{pp} - Y_{pv}Y_{vv}^{-1}Y_{vp}
\end{equation*}

The model of the grid is now given by the following expression
\begin{equation*}
  i_p = S(v) = Y_{s} v
\end{equation*}

\noindent where $S$ replaces the function $G$ in order to include the effect of the droop as follows:

\begin{equation*}
	S(v) = diag(v)^{-1} (P-C\cdot(v-v_n))
\end{equation*}

\noindent in this case, $v_n$ is a reference voltage given by the secondary control and $C$ is a positive-defined diagonal matrix.    The resulting non-linear algebraic system is given by (\ref{eq:droop})

\begin{equation}
	F(v) = S(v) - Y_{s}v = 0
	\label{eq:droop}
\end{equation}

We are interested not only in solving numerically the problems $F(v)=0$ in (\ref{eq:loadflow}) and (\ref{eq:droop}), but also in analyzing the existence of the solution and the convergence of the algorithms.

\section{Newton method}

Let $F:\mathbb{R}^n \rightarrow \mathbb{R}^n$ be a differentiable vector function.  Our objective is to find a vector $v\in\mathbb{R}^n$ such that $F(v)=0$. For this, we use the Newton's method which consists in approaching $v$ from an initial point $v_0$ by applying iteratively the following sequence

\begin{equation}
  v_{k+1}=v_{k}-[DF(v_k)]^{-1}F(v_k)
	\label{eq:newton_step}
\end{equation}

where $DF(v_k)$ is the Jacobian matrix of $F$ evaluated in the point $v_k$ (subindex $k$ indicates the iteration). Convergence of the method depends on intrinsic characteristics of the function $F(v)$ which guarantees convergence.

Let us start our analysis by presenting the following classic statement \cite{hubbard}:

\begin{theorem}{(Kantorovitch's theorem in $\mathbb{R}^n$)} Let $v_0$ be a point in $\mathbb{R}^n$ and $F:\mathcal{B}_0\rightarrow \mathbb{R}^n$ a differentiable map with its derivative $[DF(v)]$ invertible.  Define

\begin{eqnarray*}
	\Delta v_0 = [DF(v_0)]^{-1} F(v_0) \\
	v_1 = v_0+\Delta v_0 \\
	\mathcal{B}_0 = \left\{ v: \left\| v-v_1 \right\| \leq \left\|\Delta v_0\right\| \right\}
\end{eqnarray*}

\noindent if the derivative $[DF(v)]$ satisfies the Lipschitz condition 

\begin{equation}
\left\| DF(v)-DF(u)\right\|\leq K\left\|v-u\right\|, \forall v,u \in \mathcal{B}
\label{eq:lip}
\end{equation}

\noindent and if the inequality

\begin{equation}
  h= \left\|F(v_0)\right\| \left\|DF(v_0)^{-1}\right\|^2 K \leq \frac{1}{2}
	\label{eq:condicion_h}
\end{equation}

\noindent is satisfied, then the equation $F(v)=0$ has a unique solution in $\mathcal{B}_0$ and Newton's methods converges to it with Newton's step (\ref{eq:newton_step}) and initial condition $v_0$.  Moreover, if $h<\frac{1}{2}$ the order of convergence is at least quadratic.
\end{theorem}

\begin{proof} See Appendix A \end{proof}

\begin{remark}

The theorem guarantees that in each Newton's iteration, the value of $v_{k+1}$ lies in a ball $\mathcal{B}_k\subset\mathcal{B}_{0}$ which contracts as depicted in Fig \ref{fig:bola_iteracion}.  Notice that $\mathcal{B}_0$ is centered at the point $v_1$.

\begin{figure}
\centering
\footnotesize
\begin{tikzpicture}[x=1mm,y=1mm]
\draw [-latex, thick]  (140:20) -- (0,0); 
\draw [-latex,thick]  (0,0) -- (30:10); 
\draw [-latex,thick]  (30:10) -- +(100:5); 
\draw[dashed] (0,0) circle (20) (30:10) circle (10) +(100:5) circle (5);
\node at (140:23) {$v_0$};
\node at (0,-2) {$v_1$};
\node at (15:10) {$v_2$};
\node at (53:14) {$v_3$};
\node at (0,-17) {$\mathcal{B}_0$};
\node at (10,-3) {$\mathcal{B}_1$};
\node at (6,8) {$\mathcal{B}_2$};
\node[rotate=-40] at (-10,10) {$\Delta v_0$};
\end{tikzpicture}
\caption{Schematic representation of the basin of attraction of each Newton's iteration}
\label{fig:bola_iteracion}

\end{figure}
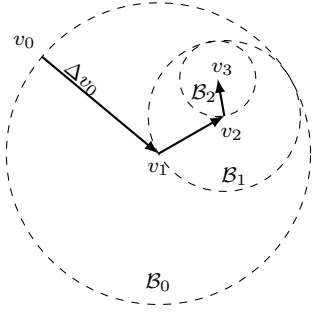	
\end{remark}

In order to use the theorem above, we consider the following assumptions easily satisfied by any practical dc-microgrid

\begin{description}
	\item[A1] The graph is connected and consequently $Y_{pp}$ is non-singular
	\item[A2] The system is not in short circuit
	\item[A3] The system is represented in per-unit.
\end{description}

In the following, we use the supremum norm for vectors $\left\|x\right\| =max\left\{|x_1|,|x_2|,\dots\right\}$ and a submutiplicative norm induced by the supremum norm for matrices

\begin{equation*}
  \left\| A \right\| = \underset{\left\|x\right\|\neq 0}{sup} \frac{\left\|Ax\right\|}{\left\|x\right\|}
\end{equation*}

Let us define some constants that will be used later: 

\begin{equation*}
\begin{array}{ll}
\rho   = \left\|Y_{pp}^{-1}\right\|              &
\alpha = \left\| P\right\|                       \\
\xi    = \left\| Y_{pv}v_{v}\right\|   &
\mu    = \left\| Y_{pv}v_v+Y_{pp}\hat{e}_N \right\|  
\end{array}
\end{equation*}

Notice that $Y_{pp}^{-1}$ is the nodal impedance matrix, then $\rho$ represents the maximum Thevening impedance of the grid while $\alpha$ is the maximum power; $\xi$ is maximum current $i_{p}$ if the constant power terminals are all in short circuit (i.e $v=0$) and finally, $\mu$ is the current for flat start initialization of the Newton method (i.e $ v=\hat{e}_{N}$ where $e_N$ is a $N\times 1$ vector of ones).

\begin{lemma}
The derivative (jacobian matrix) of $F$ in (\ref{eq:loadflow}) satisfies the Lipchitz condition for all points $v$ in an open ball $\mathcal{B} = \left\{v: \left\|v-e_N\right\|<\delta<1\right\}$
\end{lemma}

\begin{proof}
The derivative of $F$ is given by
\begin{equation*}
DF(v) = -diag(P)\cdot diag(v)^{-2}-Y_{pp}
\label{eq:jacobiano}
\end{equation*}

then, select two different points $v,u\in\mathcal{B}$ and calculate

\begin{equation*}
	\left\| DF(v)-DF(u)\right\| \leq 
	\left\| P\right\|
	\left\| diag(v)^{-2}-diag(u)^{-2} \right\|	
\end{equation*}

\begin{eqnarray*}
	\left\| DF(v)-DF(u)\right\| \leq
	\alpha \cdot \underset{i}{\text{max}} \left\{ \frac{1}{v_i^2}-\frac{1}{u_i^2} \right\} \\
	\leq \frac{2\alpha}{(1-\delta)^3} \cdot \left\|v-u \right\|
\end{eqnarray*}

then the jacobian matrix satisfies the Lipchitz condition with 

\begin{equation*}
K = \frac{2\alpha}{(1-\delta)^3}
\label{eq:m}
\end{equation*}
\end{proof}

\begin{lemma}
$DF(v)$ is invertible and its inverse is bounded for $v_0 = \hat{e}_{N}$
\end{lemma}

\begin{proof}
Define $\Gamma_0 = DF(v_0)^{-1}$ for $V_0 = \hat{e}_{N}$ (i.e flat start according to A3) and consider its inverse 
	
	\begin{equation*}
		\Gamma_0 = [DF(v_0)]^{-1} = (-diag(P)-Y_{pp})^{-1}
	\end{equation*}
	
Now, rearrange the equation as follows  
	
	\begin{equation*}
		\Gamma_0 = -Y_{pp}^{-1} (Y_{pp}^{-1}diag(P) + 1_N)^{-1}
	\end{equation*}

where $1_N$ is the $N\times N$ identity matrix.	 Notice the inverse of $Y_{pp}$ exists due to assumption A1.  In addition $\left\| Y_{pp}^{-1}diag(P)\right\|<1$ due to assumption A2; therefore, by using the Banach Lemma we can conclude that the inverse of $Y_{pp}^{-1}diag(P) + 1_N$ exist and is bounded as follows
	
	\begin{equation*}
		\left\| (Y_{pp}diag(P) + 1_N)^{-1} \right\| \leq  \frac{1}{1-\left\| Y_{pp}^{-1}diag(P)\right\|}
	\end{equation*}

\noindent by using our previously defined constants we have that
	
	\begin{equation*}
		\left\| \Gamma_0 \right\| \leq \frac{\rho}{1-\alpha\rho}
	\end{equation*}

\noindent which completes the proof.
\end{proof}

Now we can present our first result about the convergence of the Newton method in dc microgrids.
\begin{proposition}[Convergence of the Newton's Method in master-slave control]
	Under the assumptions (A1 to A3) the operation point of a dc-grid is unique and can be calculated by the Newton-method with at least quadratic convergence starting from $v_0=\hat{e}_{N}$ if
	
	\begin{equation}
	\frac{\alpha\rho^2(1-\alpha\rho)(\alpha+\mu)}{(1-2\alpha\rho-\mu\rho)^3} < \frac{1}{4}
	\label{eq:nwconv}
	\end{equation}
	
	\noindent the solutions lies in the ball $\mathcal{B}_0=\left\{ v: \left\|v-\hat{e}_N\right\|< \delta \right\}$ with
	
	\begin{equation}
		 \frac{(\alpha+\mu)\rho}{(1-\alpha\rho)} \leq \delta < 1
		\label{eq:conddeltas}
	\end{equation}
	
\end{proposition}

\begin{proof}
The proof of this proposition is immediate by direct use of Theorem 1 with the Lipschitz condition given by Lemma 1 and $\left\|DF(v_0)^{-1}\right\| = \left\|\Gamma_0\right\|$ obtained from Lemma 2.  Now notice that $\left\|F(v_0) \right\|\leq (\alpha+\mu)$  which establishes the condition (\ref{eq:conddeltas}) for the ball $\mathcal{B}_0$, that is $\left\|\Gamma_0\right\|\left\| F(v_0)\right\|\leq \delta$.  Then, (\ref{eq:nwconv}) is obtained by replacing this condition in $h$. 
\end{proof}

\begin{remark} This results defines an lower estimation of the basin of attraction of the Newton's method. Equation (\ref{eq:nwconv}) depends only on the parameters of the grid ($\rho,\mu$) and can be easily calculated for each load condition, given by $\alpha$, before executing the iterative method.  This is useful for guaranteeing convergence in practical applications where the Newton's method is usually a subroutine that requires many executions.
\end{remark}

\begin{remark}
	If Condition (\ref{eq:nwconv}) is fulfilled, then we can guarantee a minimum voltage higher than $1-\delta$ with $\delta$ given by (\ref{eq:conddeltas}).  This is also useful from a practical point of view since it gives a boundary for the minimum voltage without calculating explicitly the load flow.

\end{remark}

\begin{corollary}
	Under the assumptions of Theorem 1, convergence of the Newton's method is guaranteed if the maximum power of the grid is such that $|P_{max}| \leq \alpha_m$ where $\alpha_m$ is the minimum real root of the polynomial
	\begin{equation}
	   s(\alpha) = 4\alpha\rho^2(1-\alpha\rho)(\alpha+\mu)-(1-2\alpha\rho-\mu\rho)^3 = 0 
	\end{equation}
\end{corollary}

\begin{proof}
Just clear $\alpha$ from (\ref{eq:nwconv}) when the equality is fulfilled
\end{proof}

\begin{remark} This corollary is useful for practical applications since it allows to check convergence by just comparing the maximum power of the grid with $\alpha_{m}$.
\end{remark}

Now, let us consider the performance of the Newton's method in island operation:

\begin{proposition}
	Under the assumptions (A1 to A3) the operation point of a dc microgrid in island-mode given by (\ref{eq:droop}),  is unique and can be calculated by the Newton method with at least quadratic convergence starting from $v=\hat{e}_{N}$ if $DF(v_0)$ is non-singular and
	\begin{eqnarray}
	    \xi\gamma  \leq \delta < 1 \\
	    \frac{\xi\gamma^2\alpha}{(1-\xi\gamma)^3} \leq \frac{1}{2}
	    \label{eq:mu_isla}
	\end{eqnarray}
 \noindent with
 \begin{eqnarray}   
    \alpha = \left\| P+Cv_n\right\| \label{eq:nuevo_alpha}\\
        \gamma = \left\| Y_s + diag(P+Cv_n)\right\| \label{eq:nuevo_gamma}\\
    \xi = \left\| P-C(e_n-v_n)-Y_s\hat{e}_N\right\|  \label{eq:nuevo_xi}   
 \end{eqnarray}

\end{proposition}

\begin{proof}

First, notice that the derivative of $F$ in (\ref{eq:droop}) is given by

\begin{equation*}
DF(v) = -diag(P+Cv_n)\cdot diag(v)^{-2}-Y_{s}
\end{equation*}
which fulfills the same conditions as Lemma 1 by defining $\alpha$ as (\ref{eq:nuevo_alpha}). Now, notice that the existence of the inverse of $DF$ depends on the values of $C$. Define $\Gamma_0$ as follows

\begin{equation*}
		\Gamma_0 = [DF_0]^{-1} = (-diag(P+Cv_n)-Y_{s})^{-1}
\end{equation*}

\noindent and define $\gamma$ as (\ref{eq:nuevo_gamma}) and $\xi$ as (\ref{eq:nuevo_xi}); then we can apply directly Theorem 1. 
\end{proof}

\begin{remark}
In island operation, the matrix $Y_s$ can be singular and therefore, it is not possible to apply directly Lemma 2 as in the case of the master-slave operation
\end{remark}

\section{Approximated Newton Method}

The main drawback of the Newton's method, is the calculation of the inverse of the Jacobian matrix $DF(v)$ in each iteration.  A way to solve this drawback is the use of Approximated Newton's Methods in which the Jacobian is approximated by a constant matrix in order to reduce the time consumption of each step.  The Fast Decoupled Load Flow in AC systems is an example of the application of this approach \cite{flujo_desacoplado_rapido1,flujo_desacoplado_rapido2}.  In our case, the Jacobian matrix of (\ref{eq:loadflow}) can be approximated to the Jacobian of the fist iteration:

\begin{equation}
	DF(v_0) \approx \Gamma_0^{-1} = -diag(P)-Y_{pp} 
	\label{eq:aproximadojacobiano}
\end{equation}

\noindent The resulting iteration is given by

\begin{equation}
	v_{k+1} = v_{k} - [\Gamma_0] F(v_k)
	\label{eq:approxnewtonstep}
\end{equation}

It is obvious that we cannot apply directly Theorem 1 in this case (although Lemma 1 is still valid). However, it is possible to obtain a result about convergence by using the concept of contraction mapping.

\begin{definition}
Let $\mathcal{B}=\left\{v: \left\|v-v_0\right\|\leq  \delta \right\}$ be a closed ball of $\mathbb{R}^n$, and let $T:\mathcal{B}\rightarrow \mathbb{R}^n$. Then $T$ is said to be a contraction mapping if there is an $\beta$ such that $\left\|T(v)-T(u) \right\|\leq \beta\left\|v-u \right\|$, with $0\leq \beta < 1, \; \forall\; v,u \in \mathcal{B}$. 
\end{definition}

\begin{theorem}
 If $T$ is a contraction mapping then there is a unique $v\in \mathcal{B}$ satisfying $v = T(v)$   which can be obtained by applying the iteration $v_{k+1}=T(v_k)$ starting from an initial point in $\mathcal{B}$
\end{theorem}

\begin{proof}  See Appendix B \end{proof}

\begin{corollary}
	Let $\mathcal{B}$ be a closed ball in $\mathbb{R}^n$ and let $T:\mathcal{B}\rightarrow \mathcal{B}$ be a contraction mapping that moves the center of $\mathcal{B}$ a distance at most $(1-\beta)\delta$ with $\beta$ and $\delta$ as in Definition 1, then $T$ has a unique fixed point and it is in $\mathcal{B}$
\end{corollary}

\begin{proof} \cite{sholomo}
	Let $v_0$ be the center of the ball $\mathcal{B}_0$ and $v_c$ the center of the new ball $\mathcal{B}_c$; we have that, $\left\| T(v_0)-v_c\right\|\leq (1-\beta)\delta$ is the distance in which the center is moved, then
	
	\begin{align*}	  
		\left\| T(v)-v_c\right\| &\leq \left\| T(v)-T(v_0)\right\| + \left\| T(v_0)-v_c\right\| \\
		                         &\leq  \beta \left\| v-v_0\right\| + (1-\beta)\delta \\
														 &\leq  \beta\delta + (1-\beta)\delta  = \delta 
	\end{align*}
consequently, the range of $T$ is in $\mathcal{B}_c$	
\end{proof}

\begin{corollary}
    Let $T$ be a contraction on $\mathbb{R}^n$ and suppose that $T$ moves the point $v_0$ a distance $r$. Then the distance from $v_0$ to the fixed point is at most $r/(1-\beta)$, where $\beta$ is the contraction constant.  
\end{corollary}

\begin{proof} \cite{sholomo} Let $\mathcal{B}=\left\{v:\left\| v-v_0\right\|\leq r/(1-\beta) \right\}$ and apply Corollary 2.  It implies that  the fixed point is in $\mathcal{B}$
\end{proof}

\begin{proposition}[Convergence of the Approximated Newton Method in master-slave operation]
	Under the assumptions (A1 to A3) the operation point of a dc-grid can be calculated by the Approximated Newton's method with the iteration given by (\ref{eq:approxnewtonstep}) starting from $V_{N}=\hat{e}_{N}$ if exist two values $\beta,\delta\in\mathbb{R}$ such that 
	
	\begin{align}
	\beta = \left(\frac{\rho}{1-\alpha\rho}\right)\left( \alpha + \frac{1}{(1-\delta)^2}\right) < 1 \label{eq:nwconddv1}\\
	\left(\frac{\rho}{1-\alpha\rho}\right)\left( \frac{\alpha + \mu}{1-\beta}\right) \leq \delta < 1
	\label{eq:nwconddv2}
	\end{align}

with $\rho,\alpha$ and $\mu$ as in Proposition 1.	
\end{proposition}
\begin{proof}

Define a map $T:\mathbb{R}^n \rightarrow \mathbb{R}^n$ as

\begin{equation}
	T(v) = v - [\Gamma_0] F(v)
\end{equation}

therefore, finding the point $F(v)=0$ is equivalent to find a fixed point $v=T(v)$.  In our case

\begin{eqnarray*}
    \left\| T(v) - T(u) \right\| = 
    \left\|v-u-\Gamma_0 (F(v)-F(u)) \right\| \\ \leq
    \left\| 1_N + \Gamma_0Y_{pp}\right\|\left\|v-u \right\| +
     \left\| \Gamma_0\right\|\left\|G(v)-G(u) \right\| 
\end{eqnarray*}

where $1_N$ is the $N\times N$ identity matrix.  Now notice that $\Gamma_0 DF(v_0) = 1_N$, hence

\begin{equation*}
    \left\| 1_N + \Gamma_0Y_{pp}\right\| \leq \left\| \Gamma_0 diag(P)\right\| \leq \frac{\alpha\rho}{1-\alpha\rho}
\end{equation*}

where $\alpha$ and $\rho$ are defined as in Lemma 2.   Now, notice that $G(v)$ is locally  Lipchitz with 

\begin{equation*}
    \left\| G(v) - G(u) \right\| \leq \frac{\alpha}{(1-\delta)^2} \left\| u-v\right\|
\end{equation*}

Therefore, we have that

\begin{equation*}
    \left\| T(v) - T(u) \right\| \leq \left(\frac{\alpha\rho}{1-\alpha\rho} + \frac{\rho}{(1-\alpha\rho)(1-\delta)^2}\right) \left\| u-v\right\|
\end{equation*}

Therefore, $T$ is a contraction with $\beta$ defined as (\ref{eq:nwconddv1}) with $\Gamma_0$ defined as in Lemma 1.  Finally, notice that the fist iteration of the Approximated Newton's method moves the point the same distance as in the Newton's method. Therefore, we can apply Corollary 3 with $r=\left\| \Gamma_0\right\|\left\| F(v_0)\right\|$ in order to obtain an upper bound for $\delta$, which is given by (\ref{eq:nwconddv2}).
\end{proof}

\begin{remark} Notice that this proposition guarantees convergence but not quadratic convergence.
\end{remark}

\begin{proposition}{(Convergence  of  the  Approximated  Newton method  in  island  operation)} Under the assumptions (A2 to A3) the operation point of a dc microgrid in island-mode given by (\ref{eq:droop}),  can be calculated by the Approximated Newton's from $v=\hat{e}_{N}$ if $DF(v_0)$ is non-singular and 

\begin{eqnarray}
    \beta = \gamma \left( \alpha + \frac{1}{(1-\delta)^2}\right) < 1 \\
    \gamma \left(\frac{\xi}{1-\beta}\right) \leq \delta < 1    
\end{eqnarray}

with $\alpha,\xi,\gamma$ as in Proposition 2.
\end{proposition}

\begin{proof}
The proof of this proposition follows the same steps as in Proposition 3 with the constants given in Proposition 2.
\end{proof}

\section{Numerical Example}

A numerical simulation was performed in the dc microgrid depicted in Fig \ref{fig:grafo} whose parameters are given in Table \ref{tab:dcgrid}.  The Matlab/Octave script is available in \cite{yomatlab}.   Node 1 is voltage controlled whereas nodes 3,7,10 and 14 are step nodes (hence eliminated by a Kron reduction).  The load flow was calculated in four scenarios: using the Newton's method and the approximated Newton's method, in master-slave operation and in island operation (which means the switch in \ref{fig:grafo} is opened).

\begin{figure}[tb]
\small
\centering
\begin{tikzpicture}[x=1.5mm, y = 1.5mm, thick]
\draw[fill] (0,0) circle(0.5) node[right] {1};
\draw[fill] (-10,0) circle(0.5) node[above] {2};
\draw[fill] (0,-10) circle(0.5) node[above right] {3};
\draw[fill] (-10,-10) circle(0.5) node[below] {4};
\draw[fill] (-20,-10) circle(0.5) node[below] {5};
\draw[fill] (-20,0) circle(0.5) node[left] {6};
\draw[fill] (10,-10) circle(0.5) node[above right] {7};
\draw[fill] (10,0) circle(0.5) node[right] {8};
\draw[fill] (20,-10) circle(0.5) node[right] {9};
\draw[fill] (0,-20) circle(0.5) node[right] {10};
\draw[fill] (-10,-20) circle(0.5) node[above] {11};
\draw[fill] (-20,-20) circle(0.5) node[left] {12};
\draw[fill] (-20,-30) circle(0.5) node[below] {13};
\draw[fill] (0,-30) circle(0.5) node[above right] {14};
\draw[fill] (0,-40) circle(0.5) node[right] {15};
\draw[fill] (-10,-40) circle(0.5) node[below] {16};
\draw[fill] (-10,-30) circle(0.5) node[right] {17};
\draw[fill] (-20,-40) circle(0.5) node[left] {18};
\draw[fill] (10,-30) circle(0.5) node[below] {19};
\draw[fill] (20,-20) circle(0.5) node[right] {20};
\draw[fill] (20,-30) circle(0.5) node[right] {21};

\draw (0,5) |- (-20,-40);
\draw (-10,-40) -- (-10,-30);
\draw (0,-30) -- (20,-30);
\draw (10,-30) --(20,-20);
\draw (0,-20) -- (-20,-20);
\draw (-10,-20) -- (-20,-30);
\draw (20,-10) -- (-20,-10);
\draw (-10,-10) -- (-20,0);
\draw (10,-10) -- (10,0);
\draw (0,0) -- (-10,0);
\draw (0,5) -- +(2,2);
\draw (0,6) -- (0,10);
\draw (-2.5,10) rectangle +(5,5);
\draw (-2.5,10) -- +(5,5);
\node at (-1,13.5) {ac};
\node at ( 1,11.5) {dc};
\node at (8,12) {Master ($v$)};
\node at (5,5) {switch};
\end{tikzpicture} 
\caption{Graph of a 21 nodes dc microgrid}
\label{fig:grafo}
\end{figure}
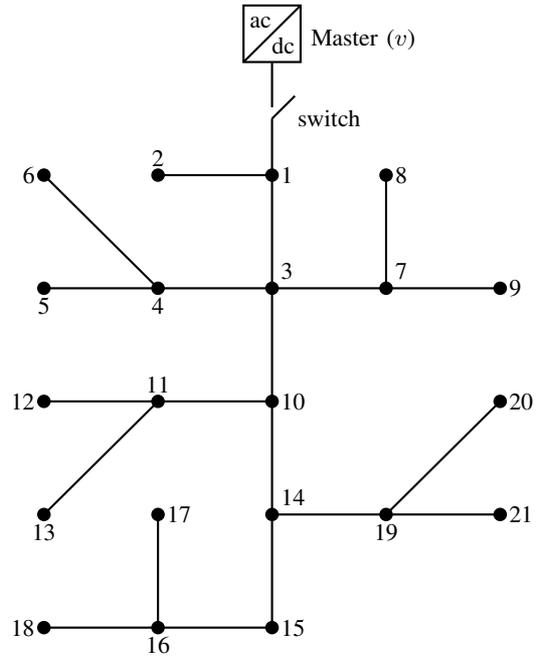

\begin{table}[tbh]
\centering
\caption{Parameters of the proposed test system}
\label{tab:dcgrid}
\begin{tabular}{|r|r|r|r|r|}
\hline
From & To & $r(pu)$ & $P(pu)$ & $1/C_{droop}$ \\
\hline\hline
	1  &  2 & 0.0053 & -0.70 &  0.05 \\
    1  &  3 & 0.0054 &  0.00 &  0.00 \\
    3  &  4 & 0.0054 & -0.36 &  0.08 \\
    4  &  5 & 0.0063 & -0.04 &  0.06 \\
    4  &  6 & 0.0051 &  0.36 &  0.07 \\
    3  &  7 & 0.0037 &  0.00 &  0.00 \\
    7  &  8 & 0.0079 & -0.32 &  0.08 \\
    7  &  9 & 0.0072 &  0.80 &  0.07 \\
    3  & 10 & 0.0053 &  0.00 &  0.00 \\
    10 & 11 & 0.0038 & -0.45 &  0.06 \\
    11 & 12 & 0.0079 & -0.68 &  0.08 \\
    11 & 13 & 0.0078 &  0.10 &  0.05 \\
    10 & 14 & 0.0083 &  0.00 &  0.00 \\
    14 & 15 & 0.0065 &  0.22 &  0.06 \\
    15 & 16 & 0.0064 & -0.23 &  0.05 \\
    16 & 17 & 0.0074 &  0.43 &  0.06 \\
    16 & 18 & 0.0081 & -0.34 &  0.08 \\
    14 & 19 & 0.0078 &  0.09 &  0.09 \\
    19 & 20 & 0.0084 &  0.21 &  0.07 \\
    19 & 21 & 0.0082 &  0.21 &  0.07 \\
\hline		
\end{tabular}
\end{table}

The value of $\left\|F(v_k)\right\|$ is given in Fig \ref{fig:convergencia} for each iteration.    We can see that the first iteration for the Newton's method and for the approximated Newton's method is the same.   However, as the algorithm is executed, the error of the Newton's method is reduced quadratically while the error of the approximated Newton's method is reduced linearly.  This performance agrees with Propositions 1 to 4; in fact, these conclusions obtained by numeric simulations, can be obtained by direct use of these prepositions.  Tables \ref{tab:summary1} and \ref{tab:summary2} summarize these results.

\begin{table}[tb]
    \centering
    \caption{Summary of the convergence measures of Master-Slave Operation}
    \label{tab:summary1}
    \begin{tabular}{|c|c|c|}
    \hline
         Measure & Newton & Approximated \\
    \hline\hline    
    $\alpha$ & 0.8000  & 0.8000 \\ 
    $\rho$   & 0.2443  & 0.2443 \\   
    $\mu$    & 0.1231  & 0.1231 \\
    $\delta$ & 0.2803  & 0.7334 \\
    $h$      & 0.1827  &  -     \\
    $\beta$  &  -      & 0.6178 \\ 
    \hline 
                & quadratic  &  linear          \\
    Conclusion  & convergence & convergence  \\
    \hline
    \end{tabular}
\end{table}

\begin{table}[tb]
    \centering
    \caption{Summary of the convergence measures of Island Operation}
    \label{tab:summary2}
    \begin{tabular}{|c|c|c|}
    \hline
         Measure & Newton & Approximated \\
    \hline\hline    
    $\alpha$ & 19.9856  & 19.9856 \\ 
    $\gamma$ &  0.0679  & 0.0679 \\   
    $\xi$    &  0.6144  & 0.6144 \\
    $\delta$ &  0.0417  & -      \\
    $h$      &  0.0643  &  -     \\
    $\beta$  &  -       & 1.44 \\ 
    \hline 
                & quadratic  &  No guarantee      \\
    Conclusion  & convergence &   of convergence \\
    \hline
    \end{tabular}
\end{table}

It is important to notice, that conditions presented in this paper are sufficient but necessary.  It means that, if the conditions are satisfied we can guarantee convergence of the method.  However, if some condition is not satisfied it does not imply the algorithm will diverge.  This is the case of the Approximated Newton's method for island operation which, in our example, does not fullfil the conditions from Proposition 4 (see Table \ref{tab:summary2}), however, as we can see in Fig \ref{fig:convergencia} the algorithm achieves convergence.

Corollary 1 can also be used to find the maximum value of power in which we guarantee quadratic convergence for master slave operation. In this case, it is  $P_{max} = 1.2406$. This result is important in applications were the power flow is executed many times, for example in optimization problems \cite{dcloadflow}. 

\begin{figure}
\centering
\footnotesize
\begin{tikzpicture}
\begin{semilogyaxis}[width=8.5cm,height=7cm,grid=both,xlabel={iterations},ylabel={$\left\| F(v)\right\|$},ymode=log]
\addplot[black, mark=o, thick] coordinates {(0,0.796470017556558)  
                                     (1,0.000158828382884) 
                                     (2,0.000000000012215)
                                     (3,0.000000000000017)};
\addplot[black, mark=*] coordinates {(0,0.796470017556558)  
                                     (1,0.000158828382884) 
                                     (2,0.000000087443442)
                                     (3,0.000000000052993)
                                     (4,0.000000000000053)};
\addplot[black, mark=triangle, thick] coordinates {(0,0.877898964219528)  
                                     (1,0.004385437319859) 
                                     (2,0.000000320721296)
                                     (3,0.000000000000165)
                                     (4,0.000000000000010)};
\addplot[black, mark = square] coordinates {(0,0.877898964219528)  
                                     (1,0.004385437319859) 
                                     (2,0.000065337705128)
                                     (3,0.000001171717181)
                                     (4,0.000000020735150)
                                     (5,0.000000000366552)
                                     (6,0.000000000006484)
                                     (7,0.000000000000120)
                                     (8,0.000000000000012)};
\legend{Master-slave Newton, Master-slave Approx, Island Newton, Island Approx };                                     
\end{semilogyaxis}
\end{tikzpicture}
\caption{Error as function of the iterations}
\label{fig:convergencia}
\end{figure}
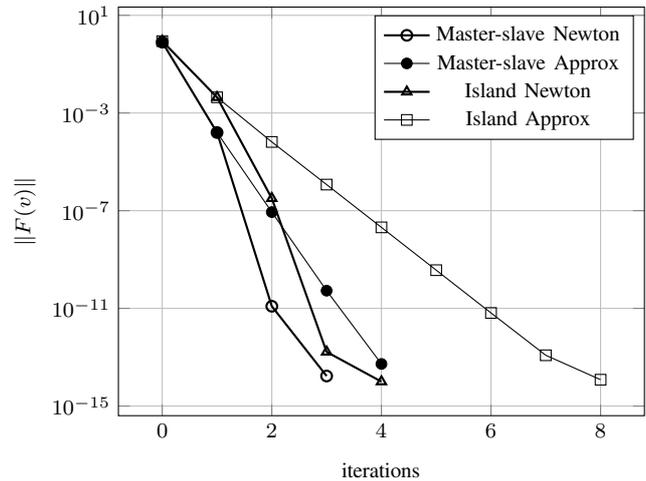

\section{Conclusions}
Exact conditions for the convergence of the the Newton's method and for the approximated newton's method were presented. The Kantorovitch's theorem and the contraction mapping theorem were used considering practical assumptions such as connectivity of the graph, per unit representation and stationary state operation. Numerical results complemented the analysis.  

The proposed analysis is important for a better understanding of the Newton method in dc grids, but also as a practical tool to determine convergence in problems in which the method is applied several times. This analysis could be extended to ac grids but more research is required.

\bibliographystyle{IEEEtran}
\bibliography{biblio}

\appendix
\subsection{Proof sketch for the Kantorovitch's theorem}

Let us define $A = I-DF(v_0)^{-1}DF(v_1)$, then replacing $I=DF(v_0)^{-1}DF(v_0)$ and by the Lipschitz condition of $DF(v)$ we have

\begin{equation*}
	\left\| A\right\| \leq \left\|DF(v_0)^{-1}\right\| K \left\| v_0-v_1\right\|
\end{equation*}

\noindent but $\left\|v_0-v_1\right\| = \left\| \Delta v_0 \right\| \leq \left\| DF(v_0)^{-1}\right\| \left\| F(v_0)\right\| $ and due to Condition (\ref{eq:condicion_h}) we have that $\left\| A \right\| \leq 1/2$.  Therefore, we can use the Banach Lemma which 
guarantees the existence of the inverse of $(I-A)$ and gives some boundaries as follows

\begin{equation*}
	 \left\| DF(v_1)^{-1}\right\| \leq \left\| DF(v_0)^{-1} \right\| \left\| (I-A)^{-1} \right\| \leq 2 \left\| DF(x_0)^{-1} \right\|
	\label{eq:restriccion_df}
\end{equation*}

On the other hand, let us define a function $g:\mathbb{R}\rightarrow \mathbb{R}^n$ as $g(t) = F(v+t\Delta v)$, then we have $g'(t) = [DF(v+t\Delta v)] \Delta t$ and hence

\begin{equation*}
	F(v+\Delta v) - F(v) = g(1)-g(0) = \int\limits_0^1 g'(t) dt
\end{equation*}

that is 
\begin{equation*}
F(v+\Delta v) - F(v) = DF(v)\Delta v + \int\limits_0^1 DF(v+t\Delta v)\Delta v - DF(v)\Delta v dt
\end{equation*}

by the Lipschitz condition of $DF$ we have

\begin{align*}
	\left\| F(v+\Delta v) - F(v) - DF(v)\Delta v\right\| &\leq \int_0^1 K \left\| v+t\Delta v - v\right\|\left\|\Delta v\right\| dt \\
	       &\leq \frac{K}{2} \left\| \Delta v\right\|^2
\end{align*}

since in each iteration $\Delta v_k = DF(v_k)^{-1} F(v_k)$ then

\begin{equation}
	\left\|F(v_{k+1})\right\| \leq \frac{K}{2} \left\| \Delta v_k\right\|^2
	\label{eq:restriccion_delta}
\end{equation}

Finally, let us analyze the step $\Delta v_1 = -DF(v_1)F(v_1)$ which by applying (\ref{eq:restriccion_df}), (\ref{eq:restriccion_delta}) and (\ref{eq:condicion_h}) we have

\begin{align*}
	\left\| \Delta v_1 \right\| &\leq \left\| DF(v_1)\right\|\left\| F(v_1) \right\|    \\
	                            &\leq (2\left\| DF(v_0)\right\|)\left(K/2\left\|\Delta v_0\right\|^2\right) \\
											        &\leq 1/2 \left\| \Delta v_0 \right\|											 
\end{align*}

By appling the same argument to the next iterations we can conclude there is a contraction of  $\Delta v$ and $F(v)$ as depicted in Fig \ref{fig:bola_iteracion}.  More details about this classic theorem can be found in \cite{hubbard}.

\subsection{Proof sketch for the Contraction mapping theorem}

The Contraction mapping theorem is general for any Banach space but we are interested only in $\mathbb{R}^n$. Let $T:\mathcal{B}\rightarrow\mathcal{B}$ be a contraction mapping in a closed ball $\mathcal{B}\in\mathbb{R}^n$, consider two point $u,v \in \mathcal{B}$ then

\begin{align*}
	\left\| u-v\right\| &= \left\| u-v + T(u)-T(v)-T(u)+T(v)\right\|  \\
	                    &\leq \left\| u-T(u)\right\| + \left\| v-T(v)\right\| + \left\|T(u)-T(v) \right\|  \\
											&\leq \left\| u-T(u)\right\| + \left\| v-T(v)\right\| + \alpha\left\|u-v \right\|  \\
\end{align*}

rearranging the inequation,

\begin{equation*}
	\left\| u-v\right\|\leq \frac{1}{1-\alpha} (\left\| u-T(u)\right\|+\left\| v-T(v)\right\|)
\end{equation*}

if $u=T(u)$ and $v=T(v)$ then $\left\| u-v\right\|\leq 0$. Since a norm is always positive except in zero, then necessarily $u=v$ which means that the fixed point is unique.  

Now define a sequence $\left\{v_k\right\}_0^{\infty}$ by the iteration $v_{k+1} = T(v_k)$.  Then it follows that

\begin{equation*}
 \left\|v_{k+n}-v_{k} \right\| \leq \left( \alpha^{k-1}\sum\limits_{m=0}^\infty \alpha^{m}\right) \left\|v_2-v_1 \right\| = \frac{\alpha^{k-1}}{1-\alpha} \left\| v_2-v_1\right\|
\end{equation*}

therefore $\left\{v_k\right\}_0^{\infty}$ is a Cauchy sequence.   Since $\mathbb{R}^n$ is complete then $\left\{v_k\right\}_0^{\infty}$ converges to a fixed $v\in\mathbb{R}^n$.  More details about this theorem can be found in \cite{sholomo}.




\end{document}